\documentclass[12pt]{article}

\usepackage{amsfonts,amsmath,amssymb,latexsym,xcolor,mathrsfs,tikz,breqn,extarrows}
\usepackage{amsthm,authblk}
\usepackage{multirow}
\usepackage{amssymb}
\usepackage{amsmath}
\usepackage{amsthm}
\usepackage{tikz}
\usepackage{mathtools}
\usepackage{calc}
\usepackage{tikz}
\usepackage{pgfplots}
\usetikzlibrary{patterns,arrows,decorations.pathreplacing}
\usepackage{diagbox}
\usepackage{makecell}
\usepackage{bbding}

\def\q{\hfill\rule{1ex}{1ex}}
\def\0{\emptyset}

\def\q{\hfill\rule{1ex}{1ex}}

\usepackage{amsthm}

\newtheorem{theorem}{Theorem}[section]
\newtheorem{definition}[theorem]{Definition}
\newtheorem{lemma}[theorem]{Lemma}

\newtheorem{cor}[theorem]{Corollary}
\newtheorem{prop}[theorem]{Proposition}

\usepackage{algorithm}
\usepackage{algpseudocode}
\usepackage{graphicx}
\usepackage{amsmath}
\usepackage{subcaption,amssymb}
\begin{document}
\title{\bf All minimum $C_4$-saturated multipartite graphs
	 }
\author[1]{
Yiduo Xu}
\author[2]{
Zhen He\thanks{Corresponding author. E-mail: \texttt{zhenhe@bjtu.edu.cn}}}
\author[1]{
Mei Lu}
\author[1]{
Yanzhe Qiu}

\affil[1]{\small Department of Mathematical Sciences, Tsinghua University, Beijing 100084, China}
\affil[2]{\small School of Mathematics and Statistics, Beijing Jiaotong University, Beijing 100044, China.}
\date{}

\maketitle\baselineskip 16.3pt

\begin{abstract}
	A subgraph $H$ of $G$ is said to be $F$-saturated relative to $G$, if $H$ does not contain any copy of $F$, but the addition of any edge $e$ in $E(G)\backslash E(H)$ would create a copy of $F$.
    The minimum size of an $F$-saturated graph relative to $G$ is denoted by $sat(G,F)$.
    Let $K_k^n$ be the complete $k$-partite graph with $n$ vertices in each part.
    In this paper, we determine $sat(K_4^n,C_4)$ for all $ n \geq 2$.
    Moreover, we determine all extremal configurations of $sat(K_k^n,C_4)$ for all $n\ge 2$ and $k\ge 4 $.
	\end {abstract}
	
	
\section{Introduction}

	In this paper we only consider finite, simple and undirected graphs.
    For a graph $G$, we use $V(G)$ to denote the vertex set of $G$, $E(G)$ the edge set of $G$, $|G|$ the order of $G$ and $e(G)$ the size of $G$.
    We say a graph $G$ is a $k$-partite graph, if $V(G)$ can be partitioned into $k$ independent sets. When $k=4$, we say $G$ is a quadripartite graph.
	
    For a graph $G$ and $e \not \in E(G)$, we use $G+e$ to denote the graph obtained by adding the edge $e$ into $G$.
    Similarly, let $G - e$ be the graph obtained by deleting the edge $e$ in $G$.
    For a vertex set $A$, let $G[A]$ be the subgraph of $G$ induced by $A$. 
	
	For a $k$-partite graph $G$ and $u \in V(G)$, let $N_G(u)=\{ v : uv \in E(G) \}$ be the \textit{neighbourhood} of $u$, and $d_G(u)=|N_G(u)|$  the degree of $u$.
    We use $\delta(G)$  to denote the minimum  degree of  $G$. Without confusion, we abbreviate as $N(u)$, $d(u)$ and $\delta$, respectively.
    For a vertex set $A $, let $N(A)= \bigcup \limits_{u \in A} N(u) \, \setminus \, A$.
    The \textit{distance} $d(u,v)$ between two vertices $u,v$ is the number of edges contained in the shortest path connecting $u$ and $v$.
    If $u,v$ are not in the same  component, then $d(u,v) = \infty$. Let $d(u,v)=0$ if $u=v$.
    We use $diam_p(G)$ to denote the maximum distance between two vertices that do not lie in the same part of $G$. For a positive integer $k$, let $[k]:=\{1,2,\ldots,k\}$.
	
	Given graphs $G$ and $F$, we say a subgraph $H$ of $G$ is \textit{$F$-saturated relative to $G$}, if $H$ does not contain any copy of $F$ but the addition of any edge $e$ in $E(G)\backslash E(H)$ would create a copy of $F$.
    The \textit{saturation number} of $F$ relative to $G$ is denoted by
	\begin{equation*}
	   sat(G,F)=\min \{ e(H):H  \; \text{is}  \; F\text{-saturated relative to } G \} \, .
	\end{equation*}
Let $Sat(G,F) = \{H : H \; \text{is}  \; F\text{-saturated relative to } G \text{ and } e(H)= sat(G,F) \}$.
If $G=K_n$, we abbreviate $sat(G,F)$ as $sat(n,F)$.
The first saturation problem was studied in 1964 by Erd\H os, Hajnal and Moon \cite{EHM} who proved that $sat(n,K_r)=(r-2)(n-1)-\frac{(r-2)(r-1)}{2}$.
For readers interested in the saturation problem, we refer to the survey \cite{survey}.
	
	The study of the cycle saturation problem has a rich history.
    The exact values and all extremal configurations are known for $C_3$ \cite{EHM}, $C_4$ \cite{OLL,TUZ} and $C_5$ \cite{Chen1,Chen2}.
    The exact asymptotic behavior was determined for $C_6$ \cite{LAN}.
    For $k \geq 7$, F{\"u}redi and Kim \cite{Fur} showed that $\left( 1 + \frac{1}{k+2}\right) n -1 \leq sat(n,C_k) \leq \left( 1+ \frac{1}{k-4} \right)n + \binom{k-4}{2}$ for $n \geq 2k-5$.
    Most recently, the exact asymptotic behavior was determined for $C_k$ with all fixed even integers $k\ge 28$ \cite{Moh}.
	
	

    Let $K_k^n$ be a complete $k$-partite graph with $n$ vertices in each part.
    The saturation number relative to complete partite graphs was initiated independently by Bollob\'as \cite{BOL} and Wessel \cite{WES}, who determined $sat(K_{n,n},K_{t,t})=n^2-(n-t+1)^2$.
    In particular, they proved $sat(K_{n,n}, C_4) = 2n - 1$. For results on the saturation number of a complete graph relative to a multipartite graph, Ferrara, Jacobson, Pfender and  Wenger \cite{FJP} determined $sat(K_k^n,K_3)$ for $k \geq 3$ and $n \ge 100$.
    Roberts \cite{ROB} showed that
    $sat(K_4^n,K_4) = 18n - 21$ for sufficiently large $n$. Gir\~ ao, Kittipassorn and  Popielarz \cite{GIRAO} determined the exact asymptotic of $sat(K_k^n,K_r) $ for $k \geq r \geq 3$.
    Notably, a clique of size $3$ is precisely the cycle $C_3$, the result on $K_3$-saturation directly resolves the case of triangles.
    The authors \cite{Xu} proved $sat(K_k^n,C_\ell)=kn+O(1)$ for all $\ell\ge 5 $ and $k\ge 2$.
    The situation for $\ell = 4$ differs because the presence of a pendant vertex (i.e., vertex of degree 1) among the four parts can force the existence of a $K_4$, violating the condition of being $C_4$-saturated.
    The saturation number of $C_4$ relative to the complete bipartite graph is determined in \cite{BOL,WES}.
    The authors \cite{Xu} determined $sat(K_k^{n},C_4)$ for $k=3$ and $k\ge 5$ and determined the extremal configuration for $k=3$.
    They also bounded $sat(K_4^{n},C_4)$.

\begin{theorem}\label{T19}\cite{Xu} For any $k \geq 5$ and $n \geq 1$,
    \begin{equation*}
    sat(K_k^n, C_4) \; =\; \left\lfloor  \frac{3(k-1)n -2}{2} \right\rfloor.
    \end{equation*}
Moreover, $ \lfloor \frac{9}{2}n-1 \rfloor \leq sat(K_4^n, C_4) \leq 5n-1$.
\end{theorem}

In this paper, we determine $sat(K_4^n,C_4)$ for all $ n \geq 2$.
Moreover, we determine all extremal configurations of $sat(K_k^n,C_4)$ for all $n\ge 2$ and $k\ge 4$.
The rest of this paper is organized as follows.
In Section 2 we present our main results on the saturation number of $C_4$ relative to multipartite graphs.
In Section 3 we prove these results.

\section{Main results}

Our main results are precise determinations of $Sat(K_k^{n},C_4)$ for $k\ge 4$.
To state these theorems, we first need to describe the families of all extremal configurations.

\subsection{Constructions and Main Results}

We first address the case of $k \geq 5$ and $n\ge 2$.
Let $\mathcal{Y}_k^n$ denote the family of $k$-partite graphs $G=V_1 \cup \ldots \cup V_k$ with $n$ vertices in each part such that the following properties hold\,:
\vspace{0.3em}

\noindent {\bf (1)} If $(k-1)n \equiv 1 \; (\mod  2 \,)$, then

(i) there exists $X^*=\{y_1,y_2,y_3\} \subseteq V_4 \cup \ldots \cup V_k$ such that $G[X^*]=K_3$\,;

(ii) for any $x_i \in V_i$ where $i \in [3]$, $N(x_i)=\{y_i\}$\,;

(iii) for any $u \in \left( \bigcup \limits_{i=4}^k V_i \right) \setminus X^* = : V^*$, $d(u)=2$ with $|N(u) \cap X^*|=1$ and $|N(u) \cap V^*|=1$. Moreover, $G[V^*]$ is a perfect matching such that for any $uv \in E(G[V^*]) $, we have $\emptyset \neq N(u) \cap N(v) \subseteq X^* $. \q
\vspace{0.5em}

\noindent {\bf (2) }If $(k-1)n \equiv 0 \; (\mod  2 \,)$, partition $V_3$ into $V_3^-$ and $V_3^+$ such that $|V_3^-|=n-1$ and $|V_3^+|=1$. We now rewrite $V_i^-:=V_i$ when $i \in [2]$ and $V_i^+:=V_i$ when $i \in [k] \setminus [4]$.  Then

(i) there exists $X^*=\{y_1,y_2,y_3\} \subseteq V_3^+ \cup V_4^+ \cup \ldots \cup V_k^+$ such that $G[X^*]=K_3$\,;

(ii) for any $x_i \in V_i^-$ where $i \in [3]$, $N(x_i)=\{y_i\}$\,;

(iii) for any $u \in \left( \bigcup \limits_{i=3}^k V_i^+ \right) \setminus X^* = : V^*$, $d(u)=2$ with $|N(u) \cap X^*|=1$ and $|N(u) \cap V^*|=1$. Moreover, $G[V^*]$ is a perfect matching such that for any $uv \in E(G[V^*]) $, we have $\emptyset \neq N(u) \cap N(v) \subseteq X^* $. \q
\vspace{0.3em}

Now we have our first main result.

\begin{theorem}\label{T111} For any $k \geq 5$ and $n \geq 2$, we have
\begin{equation*}
sat(K_k^n, C_4)  = \left\lfloor  \frac{3(k-1)n -2}{2}\right \rfloor .
\end{equation*}Moreover, $ Sat(K_k^n, C_4) = \mathcal{Y}_k^n$.
\end{theorem}

\begin{figure}[!ht]
    \centering
    \subcaptionbox{$H_1 \in \mathcal{F}_2^n \cap \mathcal{F}_4^n$.\label{Fig:B_1}}%
    {\begin{tikzpicture}[scale=.47]
    \draw (-10,0) arc(0:360:1cm and 3.5cm);
    \draw (-6,0) arc(0:360:1cm and 3.5cm);
    \draw (-2,-1) arc(0:360:1cm and 4.5cm);
    \draw (2,-1) arc(0:360:1cm and 4.5cm);
    \draw (1,-1) node[align=center]{$V_4$};
    \draw (-3,-1) node[align=center]{$V_3$};
    \draw (-7,0) node[align=center]{$V_2$};
    \draw (-11,0) node[align=center]{$V_1$};
    \filldraw (-6.99,-5) circle (5pt);
    \filldraw (-10.99,-5) circle (5pt);
    \draw (-7.0,-4.4) node[align=center]{$y^*$};
    \draw (-11.0,-4.4) node[align=center]{$x^*$};
    \draw (-10.99,-5)--(-6.99,-5);
    \draw (-10.5,-3)--(-6.99,-5) (-3.8,-3.6)--(-6.99,-5) (0.2,-3.6)--(-6.99,-5) (-7.5,-3)--(-10.99,-5);
    \draw[red,dashed,thick] (-2,-1)--(0,-1);
    \end{tikzpicture}}

    \medskip
    \subcaptionbox{$H_2 \in \mathcal{F}_2^n \cap \mathcal{F}_3^n$.\label{Fig:B_2}}%
    {\begin{tikzpicture}[scale=.47]
    \draw (-10,-1) arc(0:360:1cm and 4.5cm);
    \draw (-6,0) arc(0:360:1cm and 3.5cm);
    \draw (-2,0) arc(0:360:1cm and 3.5cm);
    \draw (2,0) arc(0:360:1cm and 3.5cm);
    \draw (1,0) node[align=center]{$V_4$};
    \draw (-3,0) node[align=center]{$V_3$};
    \draw (-7,0) node[align=center]{$V_2$};
    \draw (-11,-1) node[align=center]{$V_1$};
    \filldraw (-2.99,-5) circle (5pt);
    \filldraw (-6.99,-5) circle (5pt);
    \filldraw (1.01,-5) circle (5pt);
    \draw (1.4,-4.4) node[align=center]{$w^*$};
    \draw (-2.6,-4.4) node[align=center]{$z^*$};
    \draw (-7.0,-4.4) node[align=center]{$y^*$};
    \draw (1.01,-5)--(-2.99,-5) (-6.99,-5)--(-2.99,-5) (1.01,-5) arc(0:-180:4cm and 1cm) (-6.99,-5);
    \draw (-10.2,-3.6)--(-6.99,-5) (-3.5,-3)--(-6.99,-5) (0.5,-3)--(-6.99,-5) (-6.5,-3)--(-2.99,-5);
    \draw[red,dashed,thick] (-2,0)--(0,0);
    \end{tikzpicture}}

    \medskip
    \subcaptionbox{$H_3 \in \mathcal{F}_3^n$.\label{Fig:B_3}}%
    {\begin{tikzpicture}[scale=.47]
    \draw (-10,-1) arc(0:360:1cm and 4.5cm);
    \draw (-6,0) arc(0:360:1cm and 3.5cm);
    \draw (-2,0) arc(0:360:1cm and 3.5cm);
    \draw (2,0) arc(0:360:1cm and 3.5cm);
    \draw (1,0) node[align=center]{$V_4$};
    \draw (-3,0) node[align=center]{$V_3$};
    \draw (-7,0) node[align=center]{$V_2$};
    \draw (-11,-1) node[align=center]{$V_1$};
    \filldraw (-2.99,-5) circle (5pt);
    \filldraw (-6.99,-5) circle (5pt);
    \filldraw (1.01,-5) circle (5pt);
    \draw (1.4,-4.4) node[align=center]{$w^*$};
    \draw (-2.6,-4.4) node[align=center]{$z^*$};
    \draw (-7.0,-4.4) node[align=center]{$y^*$};
    \draw (1.01,-5)--(-2.99,-5) (-6.99,-5)--(-2.99,-5) (1.01,-5) arc(0:-180:4cm and 1cm) (-6.99,-5);
    \draw (-10.52,-5) arc(-180:0:5.77cm and 2cm) (1.02,-5) (-3.5,-3)--(-6.99,-5) (0.5,-3)--(-6.99,-5) (-6.5,-3)--(-2.99,-5);
    \draw[red,dashed,thick] (-2,0)--(0,0);
    \end{tikzpicture}}
    \caption{Examples of graphs in $\mathcal{F}^n$.}
    \label{Fig:Extremal_2-connected_Const}
\end{figure}

\vspace{0.3em}

For $k = 4$, $n \geq 2$ and $ i \in [4]$, let $\mathcal{F}_i^n$ denote the family of quadripartite graphs $G=V_1 \cup \ldots \cup V_4$ with $n$ vertices in each part such that the following properties hold\,:
\vspace{0.3em}

\noindent {\bf (1)} $G \in \mathcal{F}_1^n$ if\,:

(i) there exists $X^*=\{y_1,y_2,y_3\} \subseteq V_2 \cup V_3 \cup V_4$ such that $G[X^*]=K_3$\,;

(ii) there exists $\emptyset \neq I_i \subseteq V_i$ where $ i \in [3]$ with $|I_1|+|I_2|+|I_3|=2n-1$ such that for any $x_i \in I_i$ , $N(x_i)=\{y_i\}$\,;

(iii) for any $u \in \left( \bigcup \limits_{i=1}^4 V_i \right) \setminus \left( \bigcup \limits_{i=1}^3 I_i \cup X^* \right) = : V^*$, $d(u)=2$ with $|N(u) \cap X^*|=1$ and $|N(u) \cap V^*|=1$. Moreover, $G[V^*]$ is a perfect matching such that for any $uv \in E(G[V^*]) $, either $ u \in V_4$ or $v \in V_4$ and we have $\emptyset \neq N(u) \cap N(v) \subseteq X^* $. \q
\vspace{0.5em}

\noindent {\bf (2)} $G \in \mathcal{F}_2^n$ if\,:

(i) there exists $y^* \in V_2$ such that
$N(y^*)=V_1\cup V_3\cup V_4$ and $G[V_3 \cup V_4] $ is a perfect matching;

(ii) for any $u \in V_2 \setminus \{y^*\}$, $d(u)=1$ where $N(u)\subseteq V_1 \cup V_3 \cup V_4$. \q
\vspace{0.5em}

\noindent {\bf (3)} $G \in \mathcal{F}_3^n$ if\,:

(i) there exist $y^* \in V_2$, $\{z^*,w^*\} \subseteq V_3 \cup V_4$  such that for any $u \in V_2 \setminus \{y^*\}$, $N(u)=\{z^*\}$ and $G[\{y^*,z^*,w^*\}]=K_3$;

(ii) $V_3\cup V_4\subseteq N(y^*)$ and  $G[V_3 \cup V_4] $ is a perfect matching;

(iii) for any $u \in V_1$, $d(u)=1$ and $N(u) \subseteq \{y^*,w^*\}$.
\q
\vspace{0.5em}

\noindent {\bf (4)} $G \in \mathcal{F}_4^n$ if\,:

(i) there exist $x^* \in V_1$ and $y^* \in V_2$  such that $V_2 \subseteq N(x^*)$ and  $V_1 \subseteq N(y^*)$;

(ii)  $G[V_3 \cup V_4] $ is a perfect matching. For any $v \in V_i $ with $i \in \{3,4\}$, $N(v)=\{y^*,v'\}$ for some $v' \in V_{7-i}$ or $N(v)=\{x^*,v'\}$ for some $v' \in V_{7-i}$
\q
\vspace{0.5em}

Notice that $\mathcal{F}_2^n \cap \mathcal{F}_3^n \neq \emptyset$ and $\mathcal{F}_2^n \cap \mathcal{F}_4^n \neq \emptyset$ as shown in Figure 1.
Let $\mathcal{F}^n=\mathcal{F}_1^n \cup \mathcal{F}_2^n \cup \mathcal{F}_3^n \cup \mathcal{F}_4^n$.

Now we have our another main result. 
\begin{theorem}\label{T112} For $n \geq 2$, we have 
\begin{equation*}
sat(K_4^n, C_4) =5n-1 .
\end{equation*}Moreover $ Sat(K_4^n, C_4)=\mathcal{F}^n$. 
\end{theorem}

It is straightforward to verify that all graphs in $\mathcal{Y}_k^n$ ($k\ge 5$) are $C_4$-saturated relative to $K_k^{n}$ with $\lfloor  \frac{3(k-1)n -2}{2} \rfloor$ edges and all graphs in $\mathcal{F}^n$ are $C_4$-saturated relative to $K_4^{n}$ with $5n-1$ edges. For brevity and readability, we omit these verifications.
\vspace{0.5em}

\subsection{Properties on the lower bound}

Let $G$ be a $C_4$-saturated $k$-partite graph with $n \geq 2$ vertices in each part.  We would restate and improve some of the properties which have been proven in \cite{Xu}.

\begin{lemma}\label{L21} \cite{Xu}
Let $u$ and $v$ be pendant vertices in distinct parts of $G$, with neighbors $u'$ and $v'$, respectively.
Then $u'\neq v'$, $u'v'\in E(G)$, and at most three parts contain pendant vertices.
Clearly, $G$ is connected and $diam_p(G) \leq 3$.
\end{lemma}

Let $\ell$ be the number of parts of $G$ containing pendant vertices.
By Lemma \ref{L21} we know that $0 \leq \ell \leq 3$, and the total number of pendant vertices is at most $\ell n$.
Suppose $x_1,x_2$ are pendant vertices in the same part $i$ which have distinct neighbors $y_1$ and $y_2$. Replacing the edge $x_2y_2$ with $x_2y_1$ yields another $C_4$-saturated $k$-partite graph with the same size.
Therefore, we may assume all pendant vertices in the same part share a common neighbor. Let $I=\{v\in V(G):d(v)=1\}$.

Proposition \ref{P22} is a straightforward refinement of Proposition 4.9 in \cite{Xu}, obtained by sharpening the bound on $|A_3|$ from $2n$ to $n_2 + \dots + n_\ell$. The proof's approach, including its notation, forms the foundation of our later analysis. A sketch of proof is provided for completeness.

\begin{prop}\label{P22}
If $\delta(G)=1$, let $V_1 , \ldots ,V_\ell$ be the parts containing pendant vertices. Assume that $|V_i \cap I|=n_i$ with $n \geq n_1 \geq \cdots \geq n_\ell >0$, then $e(G) \geq   \frac{(3k-1)n-(n_2+\cdots +n_\ell) -3}{2} \geq \frac{(3k-\ell)n -3}{2}$.
\end{prop}

\noindent{\bf Proof. }Let $x \in V_1\cap I$   and $N(x)=\{y\}$. Let $S_i = \{ u \in V(G): d(x,u)=i \, \}$.
By Lemma \ref{L21}, $diam_p(G) \leq 3$. Then $V(G)=\{x\}\cup S_1\cup S_2\cup S_3\cup S_4$, where $S_1=\{y\}$ and  $S_4 \subseteq V_1$.
Let $A_2= S_2 \cap V_1$ and $B_2=S_2 \setminus A_2$. By our assumption, all pendant vertices in $V_1$ would connect to $y$ which implies $d(u) \geq 2$ for any $u \in S_4$.
By Lemma \ref{L21},  $I\setminus V_1\subseteq S_3$. Let $A_3=\{u \in S_3 : d(u)=1\}$ and $B_3=S_3 \setminus A_3$. Then $|A_3| \leq n_2+\cdots + n_\ell$ and  $|S_4|+|A_2| \leq |V_1|-1 = n-1$.

For  $u \in V(G)\setminus\{ x,y\}$, let $M_i(u)=N(u) \cap S_i$ for $i \in [4]$. Clearly if $u \in S_j$ for some $j$, then we have $M_i(u)=\emptyset$ when $|i-j| \geq 2$. Define a weight function $f$ as\,:
\begin{equation*}
f(u)=\left\{
\begin{aligned}
 \; & |M_1(u)| + \frac{1}{2} |M_2(u)| \, , \; & u \in S_2 \, ,\\
 \; & |M_2(u)| + \frac{1}{2} |M_3(u)| + \frac{1}{2}| M_4(u)| \, , \; & u \in S_3 \, , \\\
 \; & \frac{1}{2} |M_3(u)| \, , \; & u \in S_4 \, .
\end{aligned}
\right.
\end{equation*}
Then $\sum_{u \in V(G)\setminus\{ x,y\}} f(u) = e(G)-1$.
\vspace{0.5em}

 Let $u \in B_2$. Then $ux\notin E(G)$ which implies  $u$ and $x$ are connected by a $P_4$ by $G$ being $C_4$-saturated. Then $|N(u)\cap S_2|\ge 1$ and hence $|M_1(u)|+\frac{1}{2}|M_2(u)| \geq \frac{3}{2}$.
 Let $u \in B_3$. Then $d(u) \geq 2$ and we have either $|M_2(u)| \geq 2$ or $|M_2(u)|=1 $, and $|M_3(u)| + |M_4(u)| \geq 1$. Hence for any $u \in B_2 \cup B_3$, $f(u) \geq \frac{3}{2}$. If $u \in A_2$, then $f(u)\ge|M_1(u)|\ge 1$. If $u\in A_3$, then $f(u)\ge|M_2(u)|= 1$.  If $u \in S_4$, then $N(u)\subseteq S_3$ which implies $f(u)\ge \frac{1}{2}|M_3(u)| \geq 1$
 by $d(u) \geq 2$. Hence for any $u \in A_2 \cup A_3 \cup S_4$, $f(u) \geq  1$. Now we have
 \begin{equation*}
\begin{aligned}
e(G) - 1 & = \sum_{u \in V(G)\setminus\{ x,y\}} f(u)
= \sum_{u \in B_2 \cup B_3} f(u) + \sum_{u \in A_2 \cup A_3 \cup S_4} f(u) \\
& \ge \frac{3}{2}(kn-2) - \frac{1}{2} (|A_2|+|A_3|+|S_4|)  \\
& \geq \frac{3}{2}(kn-2) - \frac{1}{2}(n_2+\cdots +n_\ell+ n-1) \\
& =\frac{(3k-1)n-(n_2+\cdots +n_\ell) -5}{2} \geq  \frac{(3k-\ell)n-5}{2}
\end{aligned}
\end{equation*}
and  we are done.\qed
\vspace{0.3em}

\begin{prop}\label{P23} \cite{Xu}
If $\delta(G)=2$ and there exists a vertex of degree 2 not in a triangle, then we have $e(G) \geq \frac{(3k-1)n-4}{2}$.
\end{prop}
\vspace{0.3em}

\begin{lemma}\label{L24} \cite{Xu}
Let $H$ be a connected multipartite graph. If each edge of $H$ is in a triangle, then $e(H) \geq \frac{3}{2}(|H|-1)$.
\end{lemma}
\vspace{0.3em}



\begin{prop}\label{P25}
If $\delta(G)=2$ and each vertex of degree 2 is in a triangle, then $e(G) \geq e_k(n)$, where
\begin{equation*}
    e_k(n) := \left\{
    \begin{aligned}
    & \; 5n \, , & \; & k =4\, ; \\
    &  \frac{(3k-3)n}{2} \, , & \; &  k \geq 5   \, . \\
    \end{aligned}
    \right.
    \end{equation*}
\end{prop}
\noindent {\bf Proof. } Let $W=\{v\in V(G):d(v)=2\}$.
We first prove that $e(G) \geq  \frac{(3k-3)n}{2}$ for all $k \geq 4$.

Let $H$ be the subgraph of $G$ that consists of all triangles contained at least one vertex  in $W$, and let $H_1,\dots,H_m$ be the components of $H$.
Let $u \in V(H_i)\cap W$ and $v \in V(H_j)\cap W$ if $m\ge 2$, where $i,j\in\{1,\ldots,m\}$ and $i\not=j$. If $u$ and $v$  belong to different parts of $G$, then $u$ and $v$ are connected by a $P_4$ by $G$ being $C_4$-saturated which implies
there is at least one edge between $H_i$ and $H_j$.

For each $i \in [k]$, let $m_i=|\{j:W\cap V(H_j)\subseteq V_i, j\in [m]\}|$. Let $m^{*}=m_1+\cdots+m_k$.
Clearly, $0 \le m_i \le n$ for $1\le i\le k$, $m^{*} \le m$, and there are at most $M = \sum_{i=1}^k \binom{m_i}{2}$ pairs of components of $H$ which are non-adjacent in $G$.
By Lemma \ref{L24}, we have $e(H_i) \geq \frac{3}{2}(|H_i|-1)$ for $1\le i\le m$. Hence
\begin{equation*}
\begin{aligned}
      e(G[V(H)]) \geq & \sum_{i=1}^m \frac{3}{2}(|H_i|-1) + \binom{m}{2} - \sum_{1 \leq i \leq k} \binom{m_i}{2}  \\
      = &  \sum_{i=1}^m \frac{3}{2}|H_i| + \binom{m}{2}   - \sum_{1 \leq i \leq k} \binom{m_i}{2}  -\frac{3m}{2}=\,: \mathscr{T}.
\end{aligned}
\end{equation*}

Assume that $m=rn+s$ for some non-negative integer $r \geq 0$ and $ 0 \leq s \leq n-1$. Since $m_1+...+m_k=m^* \leq m$, $m_i \leq n$ and the combinatorial binomial is a convex function, we have
\begin{equation*}
\begin{aligned}
\binom{m}{2} - \sum_{1 \leq i \leq k} \binom{m_i}{2}-\frac{3m}{2} & \geq \frac{(rn+s)(rn+s-1)}{2} - r \cdot \frac{n(n-1)}{2}-\frac{s(s-1)}{2} - \frac{3(rn+s)}{2} \\
& = \frac{(r^2-r)n^2-3rn+(2rn-3)s}{2} =:\mathscr{R} \, .
\end{aligned}
\end{equation*}

If $r=0$, then we have $\mathscr{R}= - \frac{3}{2}s > - \frac{3}{2}n$. If $r \geq 1$, we have (note that $2rn-3 \geq 0$ as $n \geq 2$ and $r \geq 1$)
\begin{equation*}
\begin{aligned}
& \mathscr{R} \geq \frac{(r^2-r)n^2-3rn}{2}  \geq \frac{2(r^2-r)n-3rn}{2} = \frac{(2r^2-5r)n}{2} \geq -\frac{3}{2}n.
\end{aligned}
\end{equation*}
Hence we have $\mathscr{T} = \sum_{i=1}^m \frac{3}{2}|H_i|+ \binom{m}{2} - \sum_{1 \leq i \leq k} \binom{m_i}{2} -\frac{3m}{2} \geq  \frac{3}{2}|H|-\frac{3}{2}n$.
	Since  $W\subseteq V(H)$,  for any $u \in V(G) \setminus V(H)$, we have $d(u) \geq 3$. So we have $e(G) \geq e(G[V(H)])+ \frac{3}{2}(kn-|H|)\ge \mathscr{T}+ \frac{3}{2}(kn-|H|) = \frac{(3k-3)n}{2}$.

    Next we show that $e(G) \geq 5n$ when $k =4$. Suppose  $e(G) \leq 5n-1$. Since $G$ is connected, $|G|=4n$ and $e(G) \leq 5n-1$, the number of triangles containing at least one vertex  in $W$ is at most $n$ which means $m \leq n$.
    Since $\sum_{u \in V(G)} d(u) \leq 10n-2$,  $|W|\ge 2n$. Hence $m_i < m$ for any $1\le i\le k$ and then
    \begin{equation*}
        \begin{aligned}
        \binom{m}{2} - \sum_{1 \leq i \leq k} \binom{m_i}{2}-\frac{3m}{2} & \geq \binom{m}{2} - \binom{m-1}{2}-\frac{3m}{2} \\
        & = -\frac{1}{2}m-1 \geq -n \,.
        \end{aligned}
    \end{equation*}
    By the same argument as above, we have $e(G) \geq \frac{3}{2}(4n-|H|) + \frac{3}{2}|H|-n =5n$ which is a contradiction. \qed

\vspace{0.8em}

\section{Proofs of the Main Results}

    In this section we consider $k \geq 4, n \geq 2$.
    Let $G$ be a $C_4$-saturated $k$-partite graph with $n$ vertices in each part such that $e(G)=sat(K_k^n,C_4)$. By the constructions we know that
\begin{equation*}
e(G) \leq  f_k(n) := \left\{
\begin{aligned}
& \; 5n-1 \, , & \; & k =4\, ; \\
& \left\lfloor \frac{3(k-1)n-2}{2} \right\rfloor \, , & \; &  k \geq 5   \, . \\
\end{aligned}
\right.
\end{equation*}
\vspace{0.0em}
Let  $I=\{v\in V(G):d(v)=1\}$ and $\ell=|\{j:I\cap V_j\not=\emptyset,j\in[k]\}|$.
\vspace{0.3em}

\begin{lemma}\label{L31} For $k=4$ we have $2 \leq \ell \leq 3$, and for $k \geq 5$ we have $\ell =3$. Moreover, if $\ell =3$, then all pendant vertices in the same part of $G$ must have the same neighbor.
\end{lemma}
\noindent{\bf Proof. } If $\delta(G) \ge 2$, then Propositions \ref{P23} and \ref{P25} yield $e(G) \ge \min\{\frac{(3k-1)n-4}{2}, e_k(n)\} \geq f_k(n)$, where the equality holds iff $k=4$ and $n=2$. The authors have checked that there is no $C_4$-saturated quadripartite graph $G$ with $n=2$, $\delta(G)=2 $ and $e(G)=9$.
Thus, we must have $\delta(G)=1$. By Proposition \ref{P22}, this implies $e(G) \ge \frac{(3k-\ell)n-3}{2}$. Consequently, $\ell$ satisfies $2 \le \ell \le 3$ when $k=4$, and $\ell = 3$ when $k \ge 5$.

Assume $\ell=3$. By symmetry we can assume that $u_1 \in V_1\cap I$, $v \in V_2\cap I$ and $w \in V_3\cap I$ such that $u_1u_1^*,vv^*,ww^* \in E(G)$.
Suppose there exists $u_2 \in V_1$ such that $u_2u_2^* \in E(G)$ with $u_2^* \neq u_1^*$. Since $u_i$ and $z$ are connected by a $P_4$ for $i \in [2]$ and $z \in \{v,w\}$, we have $u_i^*z^* \in E(G)$.
Thus $u_1^*v^*u_2^*w^* u_1^*$ is a $C_4$ in $G$, a contradiction. \qed
\vspace{0.3em}

\subsection{Non-quadripartite graph}

Let $k \geq 5$ and $n \geq 2 $. By Lemma \ref{L31} we have $\ell=3$.
Assume that  $I\subseteq V_{1} \cup V_{2} \cup V_{3}$. By Lemma \ref{L31}, we assume $N(u)=\{y_i\}$ for any $ u \in V_{i} \cap I$, where $ i\in [3]$.
Let $X^*=\{y_1,y_2,y_3\}$. By Lemma \ref{L21} we have  $G[X^*]=K_3$.
Let $n_i = |V_{i} \cap I|$ and we may assume $n \geq n_1 \geq n_2 \geq n_3 >0$.
\vspace{0.5em}

\noindent{\bf Proof of Theorem \ref{T111}.} If $|I| \leq 3n-2$, then $n_2+n_3 \leq 2n-2$.
By Proposition \ref{P22} we have $e(G) \geq \frac{(3k-3)n-1}{2} > f_k(n)$, a contradiction.
Therefore, $|I| \geq 3n-1$ and $n_1=n_2=n$. Pick $x_i \in V_i \cap I$.
Following the notation of Proposition \ref{P22}, we let $x = x_1$ and $y = y_1$. Then $V_1\setminus\{x\}\subseteq A_2 $ by Lemma \ref{L31}. Thus $S_4 = \emptyset$.
Suppose there is $u \in V(G)\setminus\{y_1\}$ such that $|N(u) \cap S_2|\, \geq 2$, say $a_1,a_2\in N(u) \cap S_2$. Then $u a_1 y a_2 u$ is a $C_4$ in $G$, a contradiction.
Therefore, for any $u \in V(G)\setminus\{y_1\}$, $|N(u) \cap S_2|\, \leq 1$.
We divide the proof into two cases.

\noindent{\bf Case 1. }$(k-1)n \equiv 1 \; (\mod \; 2 \,)$, i.e. $k$ is even and $n$ is odd.

In this case, by  Proposition \ref{P22} and $e(G)\le f_k(n)$, we have $e(G)  = sat(K_k^n,C_4) = \frac{3(k-1)n-3}{2}$ and $n_2+n_3=2n$. So  $|I|=3n$ which implies $I= V_{1} \cup V_{2} \cup V_{3}$. Clearly for any $u \in B_2 \cup B_3$ we have $f(u)=\frac{3}{2}$, and for any $u \in A_2 \cup A_3$ we have $f(u)=1$. Still it is obvious that $y_2,y_3 \in B_2$ and $N(A_3)=\{ y_2,y_3\}$.
Since $f(u)=\frac{3}{2}$ for any vertex $ u \in B_2\cup B_3$ and $S_4 = \emptyset$, each of $G[B_2]$ and $G[B_3]$ is a perfect matching where $G[B_2]$  contains the edge $y_2y_3$. Thus $N(b)\cap \{y_2,y_3\}=\emptyset$ for any $b \in B_2 \setminus \{y_2,y_3\}$.
\vspace{0.3em}

\noindent{\bf Claim 1. } Let $u,v \in B_3$ and $uv \in E(G)$. Then $N(u) \cap B_2 = N(v) \cap B_2 =\{y_i\}$ for some $i\in \{2,3\}$. \\
\noindent {\bf Proof of Claim 1. }Since $u,v \in B_3$, $u,v\in \cup_{i=4}^kV_i$. Note that $f(u)=f(v)=\frac{3}{2}$. Then $|N(u)\cap B_2|=|N(v)\cap B_2|=1$.
Let $U=N(\{u,v\}) \cap \{y_2,y_3\}$.
    If $|U|=0$, then $d(u,x_i) \geq 4$ for $i \in \{2,3\}$ (the shortest path connecting $x_i$ and $u$ is $u b y_1 y_i x_i$ for some $b \in B_2 \setminus \{y_2,y_3\}$), a contradiction to $diam_p(G) \leq 3$.
    Suppose $|U|=2$.
    Assume $N(u) \cap B_2 = \{y_2\}$ and $N(v) \cap B_2 =\{y_3\}$. Then $uy_2y_3vu$ is a $C_4$ in $G$, a contradiction.
    So we have $|U|=1$ and we can assume $N(u) \cap B_2 =\{ y_2\}$.
    If $ N(v) \cap B_2=\{z\}$ with $ z \neq y_2$, then there is no $P_4$ connecting $x_2$ and $u$ in $G$ which implies $u \in V_2$, a contradiction. \q
\vspace{0.3em}

By the discussion above, $G[B_2 \cup B_3 \setminus \{y_2,y_3\}]$ is a perfect matching.
For any $ u,v \in (B_2 \cup B_3) \setminus \{y_2,y_3\}$ with $uv \in E(G)$, we have $d(u)=d(v)=2$ and $ \emptyset \neq N(u) \cap N(v) \subseteq X^*=\{y_1,y_2,y_3\}$. Thus, $G$ must be isomorphic to some graph in $\mathcal{Y}_k^n$.
\vspace{0.8em}

\noindent{\bf Case 2. }$(k-1)n \equiv 0 \; (\mod \; 2 \,)$.

In this case, by  Proposition \ref{P22} and $e(G)\le f_k(n)$, we have $e(G)  = sat(K_k^n,C_4) = \frac{3(k-1)n-2}{2}$. By the proof of Proposition \ref{P22}, one of the following holds\,:

(i) $|I|=3n-1$\,, and $f(u)= \frac{3}{2}$ for any $u \in B_2 \cup B_3$;

(ii) $|I|=3n$, and there exists exactly one $u^* \in B_2 \cup B_3$ such that $f(u^*)=2$ and for any $u \in (B_2 \cup B_3) \setminus \{u^*\}$,  $f(u)= \frac{3}{2}$.

If (i) holds, by the same argument as the proof in Case 1, we can show that $G$ is isomorphic to some graph in $\mathcal{Y}_k^n$. We now assume that (ii) holds. Then $n_3=n$.

As $f(u^*)=2$, either (1) $u^* \in B_2$ with $|N(u^*) \cap B_2|=2$, or (2) $u^* \in B_3$ with $|N(u^*) \cap B_2|=2$ and $|N(u^*) \cap B_3|=0$, or (3) $u^* \in B_3$ with $|N(u^*) \cap B_2|=1$ and $|N(u^*) \cap B_3|=2$.
The first two cases would each lead to a copy of $C_4$, a contradiction.
Thus, only case (3) needs to be considered.
Let $u_1,u_2 \in N(u^*)\cap B_3$.
Then $N(u_1) \cap N(u_2) \cap B_2 =\emptyset$;  otherwise there is a $C_4$ in $G$. Since $n_3 =n $, we have $u^*, u_1,u_2 \not \in V_1 \cup V_2 \cup V_3$. Let $U=N(\{u_1,u_2\}) \cap \{y_2,y_3\}$.
If $|U| \, \leq 1$, assume that $N(u_1) \cap \{y_2,y_3\} = \emptyset$, then either $d(x_2,u_1) \geq 4$ (if $y_2 \not \in N(u^*)$) or $d(x_3,u_1) \geq 4$ (if $y_3 \not \in N(u^*)$) which is a contradiction to $diam_p(G) \leq 3$.
Suppose $|U|=2$. Assume that $N(u_1) \cap B_2 = \{y_2\}$ and $N(u_2) \cap B_2 = \{y_3\}$. Then $N(u^*) \cap B_2 \not \subseteq \{y_2,y_3\}$; otherwise there exists a $C_4$.
Then there exists no $P_4$ connecting $u_i$ and $x_{i+1}$ in $G$ for $i \in [2]$, a contradiction with $G$ being $C_4$-saturated. Hence there exists no extremal graph satisfying condition (ii) and we are done. \qed

\subsection{Quadripartite graph}

Let $k=4$ and $n \geq 2$. Recall that $e(G) \leq f_4(n)=5n-1$. By Lemma \ref{L31} we know that $2 \leq \ell \leq 3$. We first consider the case $\ell=3$.

\begin{prop}\label{nl32}
    When $\ell=3$, we have $e(G)=5n-1$ and $G$ is isomorphic to some graph in $\mathcal{F}_1^n$.
\end{prop}

\noindent {\bf Proof. }Assume that $I \cap V_i= I_i \neq \emptyset$ for $ i \in [3]$.  By Lemma \ref{L31}, we assume $N(u)=\{y_i\}$ for any $ u \in V_{i} \cap I$, where $1\le i\le 3$.
Let $X^*=\{y_1,y_2,y_3\}$. By Lemma \ref{L21} we have  $G[X^*]=K_3$. Let $x_i \in I_i$ for $ i \in [3]$.
Let $V^*=\left( \bigcup \limits_{i=1}^3 V_i \right) \setminus X^*$ and $V_4^* = V_4 \setminus X^*$. Then $|V_4^*|\ge n-1$. Set $ V^*_+ = \{ u \in V^* : N(u) \cap X^* \neq \emptyset\}$  and $V_-^* = V^* \setminus V^*_+ $.
\vspace{0.5em}
	
\noindent {\bf Claim 1. }For any $u \not \in X^*$, $|N(u) \cap X^*| \leq 1$.

\noindent {\bf Proof of Claim 1. }	Suppose there is $u \not \in X^*$ such that $|N(u) \cap X^*| \geq 2$, say  $y_1,y_2 \in N(u) \cap X^*$. Then  $u y_1 y_3 y_2 u$ is a $C_4$ in $G$, a contradiction. \q
\vspace{0.5em}
	
\noindent {\bf Claim 2. }For any $u \in V_4^*$, on of the following holds:

 (i) $|N(u) \cap X^*|\, =1 $, $|N(u) \cap V_+^*|\, \geq 1$ and  there  exists $v \in N(u) \cap V_+^*$ such that $ N(v) \cap X^*= N(u) \cap X^*$;
 
 (ii) $|N(u) \cap X^*| \, =0 $, $|N(u) \cap V_+^*| \, \geq 3$ and there exists distinct $v_i \in N(u) \cap V_+^*$ such that $N(v_i) \cap X^*= \{y_i\}$ for $ i \in [3]$.
	
\noindent {\bf Proof of Claim 2. }By Claim 1,  $|N(u) \cap X^*|\, \leq 1 $ for any $u \in V_4^*$.
If $N(u) \cap X^* = \{y_i\} $ for some $ i \in [3]$, then there exists a $P_4= x_i y_i v u$ connecting $x_i$ and $u$ by $G$ being $C_4$-saturated, where $v \in V_+^*$. Then $ N(v) \cap X^*= N(u) \cap X^* = \{y_i\}$. If $N(u) \cap X^* = \emptyset$, there must exist $P_4= x_i y_i v_i u$ connecting $x_i$ and $u$ for some $v_i \in V_+^*$ such that $ N(v_i) \cap X^* = \{y_i\}$ for any $ i \in [3]$. By Claim 1, $v_i\not=v_j$ for $ i,j \in [3]$ with $i\not=j$ \q
\vspace{0.5em}
	
By Claim 1, for any $u \in V_{+}^*$ we have $|N(u) \cap X^*|\, =1$. Let $ V^*_{4-} = \{ u \in V_4^* : d(u)=2\}$  and $V_{4+}^* = V_4^* \setminus V^*_{4-} $.	
By Claim 2, for any $ u \in V^*_{4-}$,  $|N(u)\cap X^*|=1$    and $ |N(u)\cap V_+^*|=1$.
\vspace{0.5em}
	
\noindent {\bf Claim 3. }For any distinct $u,v \in V_{4-}^*$, $N(u) \cap V_+^*\not=N(v) \cap V_+^*$.
	
\noindent {\bf Proof of Claim 3. }	Suppose there are $u,v \in V_{4-}^*$ such that $N(u) \cap V_+^*=N(v) \cap V_+^*=\{w\}$.
 Assume that $N(u) \cap X^* =\{y_1\} $ and  $N(v) \cap X^* =\{y_k\} $, $k\in[3]$. If $k=1$, then $w u y_1 v w$ is a $C_4$ in $G$, a contradiction. If $ k \neq 1$, then $y_1, y_k \in N(w) \cap X^*$ by Claim 2, a contradiction with Claim 1. \q
	
	We can now calculate the number of edges in $G$. Clearly $I \subseteq V_+^*$. By Claim 3 we have $|V_{+}^* \setminus I|\, \geq |V^*_{4-}|$. Recall that $|V_4^*| \, \geq n-1$. If $|I| \, \geq 2n$, by Claims 1 and 2,  we have
\begin{equation*}
\begin{aligned}
e(G)  & \geq   e(G[X^*]) + e(X^*,  V_{+}^* \cup V^*_{4-}) + e(V_{+}^*, V^*_{4-}) + e(V^* \cup X^*, V^*_{4+}) \\
& = 3 + (|I| + |V_{+}^* \setminus I|+ | V^*_{4-}| )+ | V^*_{4-}|  + e(V^* \cup X^*, V^*_{4+}) \\
& \geq 3 +|I| + 3|V^*_{4-}| + 3 |V^*_{4+}| =  3 +|I| + 3|V_4^*| \\
& \geq 3+ 2n + 3(n-1) =5n,
\end{aligned}
\end{equation*}
a contradiction with $e(G) \leq f_4(n)=5n-1$. Thus $|I| \, \leq 2n-1$.

 Consider a weight function $g$ on $G$ satisfying\,:
\begin{equation*}
g(u)=\left\{
\begin{aligned}
 \; & |N(u) \cap X^*| + \frac{1}{2} |N(u) \cap V^*| \, , \; & & u \in V^*_+ \, ,\\
 \; &  \frac{1}{2} |N(u) \cap V^*| + \frac{1}{2} |N(u) \cap V_4^*| \, , \; & & u \in V^*_- \, ,\\
 \; & |N(u) \cap X^*| +  |N(u) \cap V_+^*| + \frac{1}{2} |N(u) \cap V_-^*| \, , \; &  & u \in V_4^* \, .
\end{aligned}
\right.
\end{equation*}	
\noindent Clearly $\sum \limits_{u \not \in X^*} g(u) = e(G)-3$.
\vspace{0.5em}


\noindent {\bf Claim 4. }(1) For any $u \in I$, $g(u)=1$. (2) For any $ u \in V^* \setminus I$, $g( u ) \geq 1$. (3) For any $u \in V_{4-}^*$, $g(u)=2$. (4) For any $u \in V_{4+}^*$, $g(u) \geq \frac{5}{2}$.

\noindent {\bf Proof of Claim 4. }By the definition we know that (1) and (3) hold.
For any $u \in V_+^* \setminus I$ we have $g(u) \geq |N(u) \cap X^*| =1 $.
For any $u \in V_-^* $ we have $g(u)= \frac{1}{2} |N(u) \cap V^*| + \frac{1}{2} |N(u) \cap V_4^*| = \frac{1}{2} d(u) \geq 1$.
Thus (2) holds.
For any $u \in V_{4+}^*$, by Claim 2 we know that $|N(u) \cap X^*| +  |N(u) \cap V_+^*| \, \geq 2$ and $d(u) \geq 3$.
Thus $g(u) \geq \frac{5}{2}$ and (4) holds. \q
\vspace{0.3em}

	By Claim 4, we have
\begin{equation*}
\begin{aligned}
e(G) - 3 & = \sum \limits_{u \not \in X^*} g(u)  = \sum_{u \in V^*} g(u) + \sum_{u \in V_{4-}^*} g(u)+ \sum_{u \in V_{4+}^*} g(u)\\
& \geq |V^*| + 2 |V_{4-}^*| + \frac{5}{2} |V_{4+}^*|  = (|V^*|+|V_{4-}^*| + |V_{4+}^*| ) + (|V_{4-}^*| + |V_{4+}^*| )  + \frac{1}{2} |V_{4+}^*| \\
& = |V(G) \setminus X^*| + |V_4^*| + \frac{1}{2} |V_{4+}^*|  \geq (4n-3) + (n-1) = 5n-4.
\end{aligned}
\end{equation*}Since $e(G)\le 5n-1$, we have $e(G)=5n-1$
and the equality holds iff $V_{4+}^* = \emptyset$, $|V_4^*|=n-1$ and for any $ u \in V^*$, $g(u)=1$.  Since $g(u)=1$ for any $ u \in V_+^* \setminus I$, we have $N(u) \cap V^* = \emptyset$ and hence $N(u) \subseteq X^* \cup V_4^*$. By Claim 3, we have $|V_+^* \setminus I| = |V_4^*|$ and $G[V_4^* \cup (V_+^* \setminus I)]$ is a perfect matching such that for any $u \in V_+^* \setminus I , v \in V_4^*$ with $uv \in E(G)$,  $d(u)=d(v)=2$ and $\emptyset \neq N(u) \cap N(v) \subseteq X^*$. Since $V_{4+}^* = \emptyset$, we have $e(V_-^*, X^* \cup V_+^* \cup V_4^*) = 0$ which means $V_-^* =\emptyset$; otherwise $G$ is not connected. By $|X^*|=3$ and $|V_+^* \setminus I| = |V_4^*|=n-1$,  $|I|=2n-1$. Hence $G$ is isomorphic to some graph in $\mathcal{F}_1^n$. \qed

\vspace{0.8em}
Then we consider the case $\ell=2$. We first give some definitions.

\begin{definition}\label{D35}
    For $i \in [4]$, we say $V_i$ is unstable if there exist two pendant vertices in $V_i$ that have different neighbors; otherwise we say $V_i$ is stable. We say $G$ is stable if all parts of $G$ are stable; otherwise we say $G$ is unstable. For a graph family $\mathcal{F}$, we say $\mathcal{F}$ is stable if for any $G \in \mathcal{F}$, $G$ is stable; otherwise we say $\mathcal{F}$ is unstable.
\end{definition}

    It is obvious that $G$ has at most one unstable part; otherwise there exists a $C_4$ by Lemma \ref{L21}. One can check that $\mathcal{F}_i^n$ are unstable when $n\ge 3$, where $i \in \{2,3\}$ and $\mathcal{F}_4^n$ is stable. We now define the operation \textit{lock} and its inverse \textit{unlock}.

\begin{definition}\label{D36}
    For an unstable graph $G$ with unstable part $V_i$, the operation that locks $V_i$ in $G$ is to create a new graph $G'$ obtained by\,:\\
    (1) selecting $|I_i|-1$ special pendant vertices in $I_i$ where the remaining one is $x^*$ with $N(x^*)=\{y^*\}$,  \\
    (2) deleting all $|I_i|-1$ edges $e_1,\ldots, e_{|I_i|-1}$ relating to these $|I_i|-1$ special vertices, \\
    (3) adding edges $e'_1, \ldots e'_{|I_i|-1}$ (one edge is related to exactly one special vertex and $y^*$).\\
    Clearly, $e(G')=e(G)$ and $G'$ is still $C_4$-saturated.
\end{definition}
\vspace{0.2em}

\begin{definition}\label{D37}
    For a stable graph $G$, the operation that unlocks $V_i$ (satisfying $I_i\not= \emptyset$) in $G$ is to create a new graph $G'$ obtained by\,:\\
    (1) selecting $|I_i|-1$ special pendant vertices in $I_i$, \\
    (2) deleting all $|I_i|-1$ edges $e_1,\ldots, e_{|I_i|-1}$ relating to these special vertices, \\
    (3) adding any edges $e'_1, \ldots e'_{|I_i|-1}$ (one edge is related to exactly one special vertex) such that $G'=G - \left( \bigcup \limits_{r=1}^{|I_i|-1 } e_r\right) + \left( \bigcup \limits_{r=1}^{|I_i|-1 } e'_r\right)$ is still $C_4$-saturated.\\
    Clearly, $e(G')=e(G)$.
\end{definition}

    Notice that there may exist several different graphs $G'$ by locking or unlocking a graph $G$. Let
\begin{equation*}
    \mathcal{L}(G,i):=\{ G':G' \text{ can be obtained by locking $V_i$ in } G\}\,,
\end{equation*}
\begin{equation*}
    \mathcal{UL}(G,i):=\{ G':G' \text{ can be obtained by unlocking $V_i$ in } G\}\,.
\end{equation*}

\noindent For the sake of convenience in narration, if $G$ is unstable then we write $\mathcal{UL}(G,i) =\{G\} $ for any $i \in [4]$; if $G$ is stable then we write $\mathcal{L}(G,i) =\{G\} $ for any $i \in [4]$. Then by our definitions, $G \in \mathcal{L}(G,i) \cap \mathcal{UL}(G,i)$ whether $G$ is stable or not. For a graph family $\mathcal{F}$, let
\begin{equation*}
    \mathcal{L}(\mathcal{F}):=\bigcup_{i \in [4], G \in \mathcal{F} } \mathcal{L}(G,i)\,,
\end{equation*}
\begin{equation*}
    \mathcal{UL}(\mathcal{F}):=\bigcup_{i \in [4], G \in \mathcal{F} } \mathcal{UL}(G,i)\,.
\end{equation*}
\vspace{0.2em}


\begin{lemma}\label{L38}
Let  $G \in Sat(K_4^n,C_4)$ such that the number of parts contained pendant vertices is $2$. Then for any $G' \in \mathcal{L}(\{G\})$, the number of parts contained pendant vertices is still $2$.
\end{lemma}
\noindent {\bf Proof. }Assume $G$ has pendant vertices in $V_1 \cup V_2$ and is unstable for $V_2$. Suppose $G'$ has pendant vertices in $V_1 \cup V_2 \cup V_3$. Pick $x_i \in V_i$ such that $N_{G'}(x_i)=\{y_i\}$ for $ i \in [3]$. By the definition there is $x_2^* \in V_2\setminus \{x_2\}$ such that $N_G(x_2^*)=\{x_3\}$ in $G$.
But then the shortest path connecting $x_1$ and $x_2^*$ in $G$ is a $P_5=x_2^* x_3 y_3 y_1 x_1$, a contradiction. \qed
\vspace{0.3em}

\begin{cor}\label{C39}
    Let $\mathcal{A}\subseteq Sat(K_4^n,C_4)$ such that the number of parts contained pendant vertices is $2$. Assume $\mathcal{L}(\mathcal{A})=\mathcal{B}$. Then $\mathcal{B} \subseteq \mathcal{UL}(\mathcal{B})= \mathcal{A}$.
\end{cor}
\noindent {\bf Proof. }It is clear that when we unlock a stable graph $G$, the number of parts contained pendant vertices is non-increasing. By Lemma \ref{L38} we know that $\mathcal{B} \subseteq \mathcal{UL}(\mathcal{B}) \subseteq \mathcal{A}$. For any $G \in \mathcal{A}$ there exists $G' \in \mathcal{B}$ and $i \in [4]$ such that $G' \in \mathcal{L}(G,i)$. Thus $G \in \mathcal{UL}(G',i) \subseteq \mathcal{UL}(\mathcal{B})$.\qed
\vspace{0.3em}
	
	Let $ H_1 \in \mathcal{F}_2^n \cap \mathcal{F}_4^n$ (see Figure 1(a)) such that all vertices in $V_2$ are adjacent to exactly one vertex $x^* \in V_1$.
    Let $H_2 \in \mathcal{F}_2^n \cap \mathcal{F}_3^n$ (see Figure 1(b)) such that all vertices in $V_2$ are adjacent to $z^* \in V_3$.
    Let $H_3 \in \mathcal{F}_3^n$ such that all vertices in $V_1$  are adjacent to $w^*$ (see Figure 1(c)). Let $\mathcal{H}=\{H_1,H_2,H_3\}$.
    It is easy to check that $\mathcal{L}(\mathcal{F}_2^n \cup \mathcal{F}_3^n)= \mathcal{H}$ and $\mathcal{UL}(\mathcal{H})= \mathcal{F}_2^n \cup \mathcal{F}_3^n$.
    Still it is easy to check that $ \mathcal{UL}(\mathcal{F}_4^n)= \mathcal{F}_2^n \cup \mathcal{F}_4^n$ in the sense of isomorphism as $\mathcal{F}_4^n$ are stable.
    We have $\mathcal{L}(\mathcal{F}_2^n \cup \mathcal{F}_3^n \cup \mathcal{F}_4^n )= \mathcal{H} \cup \mathcal{F}_4^n $ and $\mathcal{UL}(\mathcal{H}\cup \mathcal{F}_4^n )= \mathcal{F}_2^n \cup \mathcal{F}_3^n \cup \mathcal{F}_4^n $.
    We can now return to our proof on the structure of a $C_4$-saturated quadripartite graph with $\ell=2$. Note that $H_1\in \mathcal{F}_4^n$.
\vspace{0.3em}

\begin{lemma}\label{nl38}
    When $\ell=2$ and $G$ is a stable graph in $Sat(K_4^n,C_4)$, we have $e(G) = 5n-1$ and $G$ is isomorphic to some graph $\mathcal{H} \cup \mathcal{F}_4^n$.
\end{lemma}

\noindent{\bf Proof. }Assume that $I \subseteq V_1 \cup V_2$. Let $I \cap V_i = I_i $ and $|I_i|=n_i $. Pick $x_1 \in I_1$ and $x_2 \in I_2$ such that $x_1y_1,x_2y_2 \in E(G)$, then $y_1y_2\in E(G)$ by Lemma \ref{L21}. By symmetry we only need to consider that $X^*:=\{y_1,y_2\} \subseteq V_{j} \cup V_{j+1}$ for $j \in [3]$. Since $y_1y_2 \in E(G)$ and $G$ is stable, $G[X^* \cup I]$ is a double star with edge number $|X^*|+|I|-1$.
If $|I| \geq 2n$, then $|I|=2n$. We may assume $y_1 \in V_3$ and $y_2 \in V_4$, then there exists a $P_4=x_1y_1zy_2$ connecting $x_1$ and $y_2$ for some $z  \in N(y_1) \cap N(y_2)$ which means $z \in V_1 \cup V_2$, a contradiction.
Therefore, $n_1+n_2 \leq 2n-1$. Let $U^*= V(G) \setminus (I \cup X^*)$ and $U_i^* = \{ u \in U^*: |N(u) \cap X^*|=i\}$, where $0\le i\le 2$. Then $d(u)\ge 2$ for any $u\in U^*$. Let $V_{ab}=V_a \cup V_b$ for distinct $a,b \in [4]$. Since $G$ is $C_4$-free, we have the following claims.
\vspace{0.5em}

\noindent {\bf Claim 1. }$|U_2^*|\, \leq 1$. If $u \in U_2^*$, then $N(u) \cap U_1^* = \emptyset$.


\vspace{0.5em}
	
\noindent {\bf Claim 2. }For $u,v \in U_1^*$, if $uv \in E(G)$ then $N(u) \cap X^*=N(v) \cap X^*$. Moreover, $|N(u)\cap U_1^*|\le 1$ for any $u \in U_1^*$.

\noindent {\bf Proof of Claim 2. }  Suppose there are $u,v \in U_1^*$ with $uv \in E(G)$ such that $N(u) \cap X^* \neq N(v) \cap X^*$. Then $G[\{u,v, y_1,y_2\}]$ forms a $C_4$, a contradiction. Suppose there is $u \in U_1^*$ such that $|N(u)\cap U_1^*|\ge 2$, say $v_1,v_2 \in N(u) \cap U_1^*$. Then $N(u) \cap X^*=N(v_i) \cap X^*$ for $i=1,2$. Assume $N(u) \cap X^*=\{y_1\}$. Then $uv_1y_1v_2u$ is a $C_4$ in $G$, a contradiction. \q
\vspace{0.5em}

\noindent {\bf Claim 3. }For $u \in U_1^* \cap V_{34}$, $|N(u) \cap U_1^*|\,=1$. For $u \in U_0^* \cap V_{34}$, either $|N(u) \cap U_2^*|\,=1$ or $|N(u) \cap U_1^*|\,\geq 2$.

 \noindent {\bf Proof of Claim 3. }   Let $u \in U_1^* \cap V_{34}$. Assume $N(u) \cap X^*=\{y_1\}$. By Claim 2, we have $|N(u) \cap U_1^*|\,\leq 1$. If $|N(u) \cap U_1^*|\,=0$, then there is no $P_4$ connecting $x_1$ and $u$ by Claim 1, a contradiction with $G$ being $C_4$-saturated.

   Let $u \in U_0^* \cap V_{34}$. Since $G$ is $C_4$-saturated, there exists a $P_4$ connecting $u$ and $x_i$ for $i \in [2]$. Then we have either $|N(u) \cap U_2^*|\,=1$ or $|N(u) \cap U_1^*|\,\geq 2$. \q
\vspace{0.5em}

Define a weight function $h$ on $G$\,:
\begin{equation*}
h(u)=\left\{
\begin{aligned}
 \; & |N(u) \cap U_2^*| +|N(u) \cap U_1^*| +  \frac{1}{2} |N(u) \cap U_0^*| \, , \; & & u \in U^*_0 \, ,\\
 \; & |N(u) \cap X^*| + \frac{1}{2} |N(u) \cap U_1^*|\, , \; & & u \in U^*_1 \, ,\\
 \; & |N(u) \cap X^*|\,  =2 \, , \; &  & u \in U_2^* \, .
\end{aligned}
\right.
\end{equation*}	
\noindent Clearly $\sum \limits_{u \not \in X^* \cup I} h(u) = e(G)-e(G[X^* \cup I])=e(G)-|X^*|-|I|+1$.
\vspace{0.7em}

By Claim 3 and $d(u)\ge 2$ for any $u\in U^*$, we easily have the following claims.

\noindent {\bf Claim 4. }For $u \in U_1^* \cap V_{34}$, $h(u)=\frac{3}{2}$. For $u \in U_0^* \cap V_{34}$, $h(u) \geq \frac{3}{2}$ and the equality holds iff $|N(u) \cap U_2^*|\,=1$, $|N(u) \cap U_1^*|\,=0$, $|N(u) \cap U_0^*|\,=1$.

\vspace{0.5em}

\noindent {\bf Claim 5. }For $u \in U_1^* \cap V_{12}$, $h(u) \in \{1,\frac{3}{2}\}$ and $N(u) \cap U_0^* \neq \emptyset$ if $h(u)=1$. For $u \in U_0^* \cap V_{12}$, $h(u) \geq \frac{3}{2}$ where the equality hold iff $|N(u) \cap U_2^*|+|N(u) \cap U_1^*|\,=1$, $|N(u) \cap U_0^*|\,=1$.

\noindent {\bf Proof of Claim 5. }	The first result is obvious by Claims 1, 2 and $d(u) \geq 2$. Let $u \in V_i \cap U_0^*$ with $i \in [2]$. Since there exists a $P_4$ connecting $x_{3-i}$ and $u$, we have $|N(u) \cap U_2^*|+|N(u) \cap U_1^*|\,\geq 1$. By $d(u) \geq 2$, the  result holds. \q
\vspace{0.5em}

	 By Claims 4 and 5, we have
\begin{equation}
\begin{aligned}
e(G) & = \sum \limits_{u \not \in X^* \cup I} h(u) + |X^*| + |I| -1\\
& = \sum_{u \in U^* \cap V_{12}}  h(u) +  \sum_{u \in U^* \cap V_{34}}  h(u) +  |X^*| + |I| -1\\
& \geq |U^* \cap V_{12}| + \frac{3}{2} |U^* \cap V_{34}| + |X^*| + |I| -1  = : \mathscr{U}\,.
\end{aligned}
\end{equation}
Recall that $e(G) = sat(K_4^n,C_4) \leq 5n-1$. We now consider three cases relying on where $X^*$ located at.
\vspace{0.5em}

\noindent {\bf Case 1. }$X^*:=\{y_1,y_2\} \subseteq V_{1} \cup V_{2}$.
	
In this case, $y_1 \in V_2$, $y_2 \in V_1$, $U^*_2\cap V_{12}=\emptyset$ and we have
\begin{equation*}
\begin{aligned}
e(G)\ge \mathscr{U} = 4n-1 + \frac{1}{2} |U^* \cap V_{34}| =5n-1.
\end{aligned}
\end{equation*}	
Since $e(G) \leq 5n-1$, we have $e(G)=5n-1$ and
 the equality holds iff $h(u)= 1$ for any $u \in U^* \cap V_{12}$ and $h(u)= \frac{3}{2}$ for any $u \in U^* \cap V_{34}$ in equation (1). Since $h(u)= 1$ for any $u \in U^* \cap V_{12}$, we have $U^*_0\cap V_{12}=\emptyset$ by Claim 5. Suppose there is $u\in U^*_1\cap V_{12}$ such that $h(u)=1$. By Claim 5, there is $v\in U^*_0$ such that $uv\in E(G)$. Since $U^*_0\cap V_{12}=\emptyset$, $v\in V_{34}$. Thus $h(v)= \frac{3}{2}$ by $v\in U^*_0 \cap V_{34}\subseteq U^* \cap V_{34}$. Since $u\in N(v)\cap U^*_1$, we have a contradiction with Claim 4. Hence
  $U^* \cap V_{12} =\emptyset$ which implies $X^* \cup I = V_1 \cup V_2$. Since $h(u)= \frac{3}{2}$ for any $u \in U^* \cap V_{34}$, $U^*_2\cap V_{34}=\emptyset$.
    By Claims 2 and 4,  $G[V_{34}] $ is a perfect matching and any edge $uv \in E(G[V_{34}])$ satisfies $d(u)=d(v)=2$ and $N(u) \cap X^* =N(v) \cap X^*$. Thus $G$ is isomorphic to some graph $H \in \mathcal{F}_4^n$.
\vspace{0.5em}

\noindent {\bf Case 2. }$X^*:=\{y_1,y_2\} \subseteq V_{2} \cup V_{3}$. 
	
In this case, we have $y_1 \in V_2$, $y_2 \in V_3$. Then
\begin{equation*}
\begin{aligned}
e(G)\ge \mathscr{U} = 4n-1 + \frac{1}{2} |U^* \cap V_{34}| =5n-\frac{3}{2}\,.
\end{aligned}
\end{equation*}	
Since $e(G)$ is an integer, we have $e(G) \geq 5n-1$ and thus $e(G)=5n-1$. So there exists a special vertex $u^* $ such that

	(1) $h(u)= 1$ for any $u \in (U^* \cap V_{12}) \setminus \{u^*\}$ and $h(u)= \frac{3}{2}$ for any $u \in (U^* \cap V_{34}) \setminus \{u^*\}$\,;
	
	(2) $h(u^*)=\frac{3}{2}$ if $u^* \in U^* \cap V_{12}$ or $h(u^*)=2$ if $u^* \in U^* \cap V_{4}$\,.
	
\noindent As there is a $P_4$ connecting $y_2$ and $x_1$, we have $U_2^* \neq \emptyset$.
So we have $u^* \in U_2^* \cap V_{34}$ and $h(u^*)=2$. Then $U^* \cap V_{12} =\emptyset$ by Claims 1, 4, 5 (the proof is similar as in Case 1). Hence $X^* \cup I  = V_1 \cup V_2 \cup \{y_2\}$.
We assert that $U_0^* = \emptyset$, otherwise there is an edge $uv \in V_{34}$ such that $uu^*,vu^* \in E(G)$ by Claim 4 which is a contradiction with $u^* \in  V_{4}$.
By Claims 2 and 4, $G[V_{34} \setminus \{y_2,u^*\}] $ is a perfect matching  and any edge $uv  \in G[V_{34} \setminus \{y_2,u^*\}]$ satisfies $d(u)=d(v)=2$ and $N(u) \cap X^* =N(v) \cap X^* =\{y_1\}$ since $y_2 \in V_3$.
Thus $G$ is isomorphic to $H_2$.
\vspace{0.5em}

\noindent {\bf Case 3. }$X^*:=\{y_1,y_2\} \subseteq V_{3} \cup V_{4}$.

	By symmetry we can assume $y_1 \in V_4 $, $y_2 \in V_3$. As there exists a $P_4$ connecting $y_i$ and $x_{3-i}$, we have $U_2^* \neq \emptyset$ and hence $\{u^*\}=U_2^* \subseteq V_{12}$. We now recalculate equation (1)\,:
\begin{equation}
\begin{aligned}
e(G) & = \sum \limits_{u \not \in X^* \cup I} h(u) + |X^*| + |I| -1\\
& = \sum_{u \in (U^* \cap V_{12}) \setminus\{u^*\}}  h(u) +h(u^*)+  \sum_{u \in U^* \cap V_{34}}  h(u) +  |X^*| + |I| -1\\
& \geq |(U^* \cap V_{12}) \setminus\{u^*\}| + \frac{3}{2} |U^* \cap V_{34}| + |X^*| + |I| +1 \\
& \geq 4n+ \frac{1}{2} |U^* \cap V_{34}| =5n-1\,.
\end{aligned}
\end{equation}
Thus we have $e(G)= 5n-1$ and $h(u)= 1$ for any $u \in (U^* \cap V_{12}) \setminus \{u^*\}$ and $h(u)= \frac{3}{2}$ for any $u \in U^* \cap V_{34}$. Then $(U^* \cap V_{12}) \setminus\{u^*\} =\emptyset$ by Claims 1, 4, 5 (the proof is similar as in Case 1).
Hence $V_1 \cup V_2 = I \cup \{u^*\}$. We assert that $U_1^* = \emptyset$, otherwise there is an edge $uv \in V_{34}$ such that $uy_i,vy_i \in E(G)$ for some $i \in [2]$ by Claims 2 and 4 which is a contradiction to $y_i \in  V_{34}$. Therefore $V_{34} \subseteq U_0^*$. By Claims 4, $G[V_{34} \setminus X^*] $ is a perfect matching and any edge $uv  \in G[V_{34} \setminus \{y_2,u^*\}]$ satisfies $d(u)=d(v)=2$ and $N(u) \cap N(v) =\{u^*\}$. Thus $G$ is  isomorphic to $H_3$.
\vspace{0.5em}

	By the discussion above, we know that $e(G)=5n-1$ and $G$ is isomorphic to either $H_2$ or $H_3$ or some $H \in \mathcal{F}_4^n$ and we are done. \qed
	
%

\vspace{0.3em}

\begin{cor}\label{C311}
    Let $\mathcal{A}\subseteq Sat(K_4^n,C_4)$ such that the number of parts contained pendant vertices is $2$. Then $\mathcal{A}=\mathcal{F}_2^n \cup \mathcal{F}_3^n \cup \mathcal{F}_4^n $.
\end{cor}
\noindent{\bf Proof. }Let $\mathcal{L}(\mathcal{A})=\mathcal{B}$. By Lemma \ref{nl38}, any graph $G \in \mathcal{B}$ must be isomorphic to some graph in $\mathcal{F}_4^n \cup \mathcal{H}$. As we have checked that any  $H \in \mathcal{F}_4^n \cup \mathcal{H}$ is $C_4$-saturated in Section 2, $\mathcal{B} = \mathcal{F}_4^n \cup \mathcal{H}$. By Corollary \ref{C39}, $\mathcal{A} = \mathcal{UL}(\mathcal{B})= \mathcal{UL}(\mathcal{H}\cup \mathcal{F}_4^n )= \mathcal{F}_2^n \cup \mathcal{F}_3^n \cup \mathcal{F}_4^n $. \qed
\vspace{1em}

\noindent{\bf Proof of Theorem \ref{T112}. }For $k = 4$ and $n \geq 2$, by Proposition \ref{nl32} and Corollary \ref{C311}, any $C_4$-saturated quadripartite graph $G$ with minimum number of edges must be isomorphic to some graph in $\mathcal{F}_1^n \cup \mathcal{F}_2^n \cup \mathcal{F}_3^n \cup \mathcal{F}_4^n =\mathcal{F}^n $. Thus $sat(K_4^n,C_4)=5n-1$. \qed

\section*{Acknowledgement}
This paper is supported by the National Natural Science Foundation of China (No.~12401445, 12571372).


\begin{thebibliography}{99}


	\bibitem{BOL} B. Bollob\'as, On a conjecture of Erd\H os, Hajnal and Moon, Amer. Math. Monthly, 74(1967), pp. 178-179.


	\bibitem{Chen1} Y. Chen, Minimum $C_5$-saturated graphs, J. Graph Theory, 61(2009), pp. 111–126.

	\bibitem{Chen2} Y. Chen, All minimum $C_5$-saturated graphs, J. Graph Theory, 67(2011), pp. 9–26.

	\bibitem{survey} B.L. Currie, J.R. Faudree, R.J. Faudree and J.R. Schmitt, A survey of minimum saturated graphs, Electron. J. Combin., 18 (2011), D519.


	\bibitem{EHM} P. Erd\H os, A. Hajnal and J.W. Moon, A problem in graph theory,  Amer. Math. Monthly  71(1964), pp. 1107-1110.

	\bibitem{FJP} M. Ferrara, M. S. Jacobson, F. Pfender and P. S. Wenger, Graph saturation in multipartite graphs, J. Comb., 7(2016), pp. 1-19.



 	\bibitem{Fur} Z. F{\"u}redi, Y. Kim, Cycle-saturated graphs with minimum number of edges, J. Graph Theory, 73(2)(2013), pp. 203-215.


	\bibitem{GIRAO} A. Gir\~ ao, T. Kittipassorn and K. Popielarz, Partite saturation of complete graphs, SIAM J. Discrete Math., 33(4)(2019), pp.2346-2359.




	
	\bibitem{LAN} Y. Lan, Y. Shi, Y. Wang and J. Zhang, The saturation number of $C_6$, Discrete Math. 348(8)(2025), 114504.

	 \bibitem{Moh} A. Mohammadiam, M. Poursoltani and B. Tayfeh-Rezaie, On saturation numbers of complete multipartite graphs and even cycles, Arxiv:2506.09767.

	\bibitem{OLL} L. Ollmann, $K_{2,2}$-saturated graphs with a minimal number of edges, Pro. 3rd Southeastern Conference on Combinatorics, Graph and Computing (1972), pp. 367-392.

		
	\bibitem{ROB} B. Roberts, Partite saturation problems, J. Graph Theory, 85(2017), pp. 429-445.



    \bibitem{TUZ} Z. Tuza, $C_4$-saturated graphs of minimum size, Acta Univ. Carolin. Math. Phys., 30(2)(1989), pp. 161-167.

    \bibitem{WES} W. Wessel, \"Uber eine Klasse paarer Graphen, I: Beweis einer Vermutung von Erd\H os, Hajnal and Moon, Wiss. Z. Hochsch. Ilmenau, 12(1966), pp. 253-256.

    \bibitem{Xu} Y. Xu, Z. He and M. Lu, Partite saturation number of cycles, Discrete Mathematics. 349(3)(2026), 114802.


	

 \end{thebibliography}
\end{document}